\documentclass[12pt,A4paper]{article}

\usepackage[left=1in, right=1in, top=1in, bottom=30mm]{geometry}

\usepackage{epsfig}
\usepackage{epstopdf}
\usepackage{soul}
\usepackage{xfrac}
\usepackage{amsmath}
\usepackage{enumerate}

\usepackage[percent]{overpic}

\usepackage{color}
\usepackage{amsthm,amsmath,amssymb,mathrsfs}
\usepackage{booktabs}
\usepackage{mathpazo}
\usepackage{microtype}
\usepackage{overpic}
\usepackage{bm}
\usepackage{sectsty}
\usepackage[
	pdftitle={PDFTitle},
	pdfauthor={Hugo Parlier},
	ocgcolorlinks,
	linkcolor=linkred,
	citecolor=linkred,
	urlcolor=linkblue]
{hyperref}

\usepackage{tikz}
\usetikzlibrary{calc,decorations.pathreplacing}

\definecolor{linkred}{RGB}{112, 11, 75} 
\definecolor{linkblue}{RGB}{16, 78, 139}

\usepackage[hang,flushmargin]{footmisc}

\usepackage{titlesec}
	\titlespacing{\section}{0pt}{12pt}{0pt}
	\titlespacing{\subsection}{0pt}{6pt}{0pt}
	
\titlelabel{\thetitle.\quad}

\makeatother 

\theoremstyle{plain}
\newtheorem{theorem}{Theorem}[section]
\newtheorem{proposition}[theorem]{Proposition}
\newtheorem{lemma}[theorem]{Lemma}
\newtheorem{corollary}[theorem]{Corollary}

\theoremstyle{definition}

\newcommand{\sys}{{\rm sys}}

\newcommand{\Hyp}{{\mathbb H}}
\newcommand{\N}{{\mathbb N}}

\newcommand{\area}{{\rm area}}
\newcommand{\arcsinh}{{\,\rm arcsinh}}
\newcommand{\arccosh}{{\,\rm arccosh}}

\newcommand{\diam}{{\rm diam}}

\newcommand{\inv}{{\rm inv}}

\newcommand{\curves}{\mathscr C}
\newcommand{\multi}{\mathscr{C^{\times}}}

\sectionfont{\large \bfseries}
\subsectionfont{\normalsize}

\long\def\symbolfootnote[#1]#2{\begingroup%
\def\thefootnote{\fnsymbol{footnote}}\footnote[#1]{#2}\endgroup}

\def\blfootnote{\xdef\@thefnmark{}\@footnotetext}

\usepackage{mathtools}

\usepackage[sort,nocompress]{cite}

\begin{document}

{\Large \bfseries Counting curves with bounded intersection}

{\large Hugo Parlier\symbolfootnote[1]{\small Supported by ANR-SNF Grant number 200021E\_238147
(SUGAR).
}
}

\vspace{0.5cm}
{\bf Abstract.}
This paper counts curves, and multicurves, on hyperbolic surfaces, providing explicit upper bounds that depend polynomially on length and exponentially in terms of an upper bound on self-intersection.

\vspace{0.5cm}

\section{Introduction}\label{sec:intro}

Curve counting, in its different forms, is a popular sport in geometric topology. In light of how they link geometric and dynamical points of view, counting closed geodesics on hyperbolic surfaces is particularly beloved. Two points of view are notable for their contrasting interpretations. The first is to consider all (primitive) closed geodesics. For closed and geodesically complete finite area surfaces, closed geodesics are dense in the unit tangent bundle, and by a theorem of Huber \cite{Huber}, as $L$ grows, the number of geodesics of length at most $L$ is asymptotic to $e^L/L$. In contrast, Birman and Series \cite{Birman-Series85} showed that restricting to simple closed geodesics, that is those without self-intersection points, paints a completely different picture. They are nowhere dense, form a Hausdorff dimension $1$ subset of the surface and only grow polynomially in length. Mirzakhani \cite{Mirzakhani2008} famously proved that the growth function of simple closed geodesics was in fact asymptotic to a polynomial of degree the dimension of the underlying moduli space, with a leading term that depends on the geometry of the surface. So closed geodesics are everywhere, whereas simple ones are special. Understanding closed geodesics is key to understanding the geodesic flow whereas understanding simple ones is intimately related to growth in the mapping class group, the group of self-homeomorphisms of the surface up to isotopy, a group of great importance.

A natural way to interpolate between these two extremes is to index subsets of closed geodesics by their number of self-intersections $K$. For fixed $K$, the results of Birman and Series and Mirzakhani continue to hold. Mirzakhani identified a fascinating phenomenon: the number of closed geodesics with at most $K$ self-intersections grows at the same rate as the simple ones, only with a additional leading multiplicative factor that only depends on $K$, and not on the surface $X$. In other directions, Sapir \cite{Sapir2021} considered sets of all geodesics with at most $K$ self-intersection points where $K$ grows as a function of $L$, showing that the non-density and Hausdorff dimension $1$ phenomenon persists far beyond fixed intersection. To do so, she proved bounds on the number of closed geodesics, building on previous work in \cite{Sapir2016}. 

A particularly interesting case is when the number of self-intersections is bounded by $\varepsilon L^2$. This comes from the fact that a closed geodesic chosen uniformly at random among all those of length at most $L$ on a surface $X$ will typically have on the order of $L^2$ self-intersection points. More precisely, Lalley \cite{Lalley} showed that most such geodesics have roughly $\kappa L^2$ self-intersection points where $\kappa$ just depends on the topology of $X$ (see Section \ref{sec:corollaries}). Torkaman \cite{Torkaman} considered this point of view to analyze the behavior of the number of such geodesics for small $\varepsilon$, namely showing that as $\varepsilon$ goes to $0$, this number loses its exponential growth. This can be expressed in terms of a type of entropy. The entropy of a set of geodesics, indexed by length, is the limit, in terms of length, of the quotient of the logarithm of the cardinality of the set divided by the length. Torkaman shows that the set of geodesics of length at most $L$ and with self-intersection at most $\varepsilon L^2$ has entropy tending to $0$ as $\varepsilon$ tends to $0$ by showing that this limit is bounded above by $C_X \sqrt{\varepsilon} \log \frac{1}{\varepsilon}$ where $C_X$ is a function that depends on the surface $X$. 

The goal of this paper is to show a related general and unifying result on the number of closed geodesics of length at most $L$ and self-intersection at most $K$. The main theorem is an upper bound on the number $N_X(L,K)$ of primitive geodesic multicurves with these properties. A precise definition of what is meant by a primitive multicurve is given in the preliminary section, but in particular, this number bounds all primitive closed geodesics. 

The first result is when the surface has boundary and the bound depends on its systole $\sys(X)$, the length of the shortest closed geodesic of $X$.  

\begin{theorem}\label{thm:boundaryintro}
Let $X$ be a hyperbolic surface with geodesic boundary. Then, for any $L>0$, we have
\[
N_X(L,K) \leq 
\left(\frac{e^{4\pi |\chi|/\sys(X)}}{2\,\sys(X)}\,L+4\right)^{12 |\chi|} e^{2\pi\sqrt{ |\chi| K}}
\]
where $\chi = \chi(X)$ is the Euler characteristic of $X$. 
\end{theorem}
A main feature is that this bound holds for all $L>0$, whereas an asymptotic result only holds for, often mysterious, large enough $L>0$.

Here is the rough idea of the proof. We first push all curves to a geometric spine without distorting length too much, using an idea from \cite{BasmajianParlier2026}. When pushed to a spine, when you forget how parts of the curve are pasted together and in which order they appear, a multicurve now looks like a simple multicurve, the number of which can be bounded by a polynomial of the length. Additional data, such as the orientation of pieces of the curve on the spine, can also be bounded by a polynomial function of length. The key to the above inequality is to now encode the rest of the data in a way that the remaining necessary data to recover the original multicurve only depends on the upper bound on the self-intersection number $K$. This is done via a type of encoding which associates binary words to the edges of the spine, and the rest of the count of words can be bounded using classical bounds on the number of multi partitions. The idea of using a word encoding for curves dates back to Poincar\'e and Dehn, but is particularly effective in the work of Birman-Series \cite{Birman-Series85}, Cohen-Lustig \cite{CohenLustig}, Chas \cite{Chas2004} and Chas-Lalley \cite{ChasLalley}. The encoding used here is also similar to one used in Torkaman's bounds \cite{Torkaman}. What makes the bound more effective is a clear separation between what data can be bounded by a polynomial, and which absorbs all of the length, and the exponential part, which only depends on the topology and the self-intersection number. This separation of length and intersection is reminiscent of what is called "fixed parameter tractable" for algorithms. 

All of this is possible because of the somewhat simple graph structure of the geometric spine, and there is no such spine for closed surfaces. To tackle this problem, the approach is quite different: we construct a surface with a single boundary curve and an injective map from all curves of the original closed surface to the surface with boundary. This map is done in a geometric way such that the lengths are not increased by too much, and such that the surface with boundary retains some of the geometric properties of the original surface, namely a related bound on the systole. We can then apply the first theorem to obtain the following result for closed surfaces.

\begin{theorem}\label{thm:closedintro}
Let $X$ be a closed hyperbolic surface. Then, for any $L>0$,
\[
N_X(L,K) \leq
\left(
\frac{5}{\left(s'_0\right)^2}
e^{\frac{16q+4\pi(|\chi|+1)}{s'_0}}\,L+4
\right)^{12(|\chi|+1)}
e^{2\pi\sqrt{(|\chi|+1)K}}.
\]
where $\chi= \chi(X)$, $r_0 = \frac{1}{2} \arcsinh(2/\sqrt{3})$ and $s'_0=\min\{\sys(X), 2\pi\sinh r_0\}$.
\end{theorem}

The choice of constant $r_0 \approx 0.5$ is somewhat arbitrary, and is mostly an artifact of the proof. The point is that, like for Theorem \ref{thm:boundaryintro}, when the systole is very small, the coefficients in the polynomial grow, but otherwise they are also uniformly bounded.

These results can then be applied to Torkaman's problem mentioned above, simply by replacing $K$ by $\varepsilon L^2$. Setting $h_X(\varepsilon)$ to be the entropy of multicurves of length at most $L$ with at most $\varepsilon L^2$ self-intersections, we obtain the following. 

\begin{corollary}\label{cor:epsilonintro}
For $X$ with geodesic boundary, 
\[
h_X(\varepsilon)\leq 2\pi\sqrt{ |\chi| \varepsilon}
\]
and, for closed $X$,
\[
h_X(\varepsilon)\leq 2\pi\sqrt{( |\chi| +1) \varepsilon}
\]
where $\chi= \chi (X)$.
\end{corollary}
In particular, $\limsup_{\varepsilon \to 0^+} \frac{h_X(\varepsilon)}{\sqrt{\varepsilon}} $ is bounded above by a universal constant. The $\liminf$ of the same quantity can also be bounded, this time from below by  $\pi\sqrt{2|\chi|}$ (see Proposition \ref{prop:lower}), leaving a, slightly irritating, multiplicative discrepancy of $\sqrt{2}$. The proof is a direct consequence of Lalley's theorem. This raises the question of whether the actual limit exists and, if so, whether it is equal to the lower bound given by Proposition \ref{prop:lower}. 

\noindent {\bf Acknowledgements.}

This work benefited greatly from many conversations about some of the ideas and results in earlier stages. In particular I thank Marie Abadie, Nalini Anantharaman, Ara Basmajian, David Fisac, Alfredo Hubard, Francis Lazarus, Arnaud de Mesmay and Jenya Sapir. 

\noindent {\bf Statement on the use of AI and the origin of ideas.}

While preparing the manuscript, certain standard AI tools were used for formatting, finding references and checking certain computations. It happens that AI sometimes provides more information than requested which may or may not have led to new ideas and/or extra confusion. 

Statements such as these do suggest thinking about where the ideas actually did come from. For this particular paper, there are certainly multiple origins. On the one hand, with Ara Basmajian, we recently proved a lemma in \cite{BasmajianParlier2026} which seemed useful in other contexts (this is Lemma \ref{lem:spine} below). It suggests pushing the geometry of a surface with boundary to a spine, and then working combinatorially in a spine. The research into the background on entropy as a measure of growth of geodesics, led to Torkaman's work on the entropy of currents \cite{Torkaman}. The methods of proof, and questions raised, reminded me of the work of Sapir \cite{Sapir2016,Sapir2021} who had this great idea of considering sets of geodesics with growing intersection. Ultimately, proving the counting question for curves with bounded intersection led to the more general results on $K$-curves, namely influenced by a paper by Aougab and Souto \cite{AougabSouto} who have a great conjecture on the number of topological types of $K$-curves, and the results of Mirzakhani \cite{MirzakhaniOrbits}, further pushed by Erlandsson and Souto \cite{ErlandssonSoutoCounting, ErlandssonSoutoBook}, which imply that the growth of $K$-curves is really just the same as the growth of simple curves times a topological constant. This suggests a type of separation of variables phenomenon one might hope for in an upper bound on all curves with at most $K$ self-intersections. 

\section{Setup and preliminary results}\label{sec:prelim}

Throughout $\Sigma$ will be a finite type orientable surface with $\chi(\Sigma) <0$ and $X \cong \Sigma$ will designate a hyperbolic surface homeomorphic to $\Sigma$. By focussing on its Nielsen core, we can think of the surface $X$ as having genus, cusps or boundary curves which are simple closed geodesics. By Gauss-Bonnet, the area of $X$ is
\[
\area(X)= 2\pi | \chi(X) | = 2 \pi q
\]
where we use $q=| \chi(X) |$. 

When a surface $X$ has geodesic boundary, we consider two quantities: its boundary length $\ell(\partial X)$, which is the sum of lengths of its boundary geodesics, and its boundary width $w(X)$ defined as 
\[
 w(X)=\max_{x\in X} d(x,\partial X).
\]
The width can be estimated using other invariants, such as systole length and area. For instance, an easy area argument gives 
\begin{equation}\label{eq:systolewidth}
w(X) \leq \frac{1}{2 \sinh \frac{\sys(X)}{4}} \,\area(X) = \frac{\pi}{\sinh \frac{\sys(X)}{4}} q < \frac{4\pi q}{\sys(X)}.
\end{equation}
The first inequality comes by taking a distance path between a point and the boundary and observing that a strip of width $\sys(X) / 4$ is embedded, and hence its area is less than the full area of the surface. 

In fact, the same argument shows that if $X$ is closed then its diameter $\diam(X)$ satisfies the same inequality:

\begin{equation}\label{eq:diam}
\diam(X) < \frac{4\pi q}{\sys(X)}.
\end{equation}

By curve we mean a free homotopy class of oriented closed curve, which is both essentiel (so not homotopic to a point or a boundary element) and primitive (so it not the iterate of another homotopy class). The set of curves on a surface of is denoted $\curves(\cdot)$, so for instance we can consider $\curves(\Sigma)$ and $\curves(X)$. Through the uniqueness of a geodesic in a free homotopy class, curves on $\Sigma$ correspond to non-boundary closed geodesics on $X$ so the notion of the length of a curve on $X$ makes sense. The length of the (or a) shortest closed geodesic of $X$ is its systole $\sys(X)$, where this time we do take into account boundary geodesics. 

Similarly, we can consider multicurves, which we define here to be collections of distinct curves. Note they are not necessarily simple and they are not weighted. As curves are considered primitive, and they are all distinct, these are also naturally primitive. Because we are considering oriented multicurves, there is one small additional supposition: a multicurve cannot contain both a curve and the same curve with opposite orientation. We denote by $\multi(\cdot)$ the set of multicurves. As before, we can measure lengths of elements of $\multi(X)$. 

Curves and multicurves are classified according to their intersection numbers. As there is some discrepancy between notations and definitions in the literature, here is the notation used here. 

For two curves $\alpha, \beta \in \curves(\Sigma)$, we define their intersection number $i(\alpha,\beta)$ to be the minimum number of intersection points between independently chosen representatives in general position. In particular, $i(\alpha,\alpha)=2k$ for some $k\in \N$. When $k=0$, the curve is called simple. In general, $k$ is the self-intersection number of $\alpha$, so is just the quantity $\frac{1}{2} i(\alpha,\alpha)$. 

Now for multicurves $\Gamma,\Gamma' \in \multi(\Sigma)$, we set their intersection number $i(\Gamma,\Gamma')$ to be the sum of the intersection numbers of all pairs of components. And the self-intersection of a multicurve $\Gamma$ is the quantity $\frac{1}{2} i(\Gamma,\Gamma)$. So, if $\Gamma = \alpha_1 \cup \cdots \cup \alpha_m$, this number is 
\[
\frac{1}{2} i(\Gamma,\Gamma) = \sum_{k=1}^m \frac{1}{2} i(\alpha_k,\alpha_k) + \sum_{1 \leq k < l \leq m}i(\alpha_k,\alpha_l).
\]
This is coherent: an intersection point between two distinct curves of $\Gamma$ is counted twice in the quantity $i(\Gamma,\Gamma)$, as is a self-intersection point of a curve which belongs to $\Gamma$. 

Finally, we can formally define the quantities of interest. We now fix a hyperbolic surface $X$, $L>0$, and $K\in \N$. The quantity $N_X(L,K)$ is the number of geodesic multicurves of length at most $L$ and self-intersection at most $K$. 

\section{Encoding and counting closed curves on surfaces with boundary}

\subsection{Pushing curves to a combinatorial model}
As stated previously, the counting problem splits into a geometric counting problem of simple multigeodesics, the number of which are bounded by a polynomial of length, and a combinatorial counting problem which is entirely responsible for exponential growth. 

We begin by pushing a primitive multi-curve to a well-defined spine or carrier graph of our surface. We borrow the following result from \cite{BasmajianParlier2026} which finds a geometric carrier graph and quantifies how lengths of curves are distorted when pushed to the spine. 

\begin{lemma}\label{lem:spine} Let $X$ be a surface with $\partial X \neq \emptyset$ and width $w(X)$ There is a finite embedded geodesic graph $G_X\subset X$ such that every closed geodesic $\gamma$ in $X$ is freely homotopic to a closed path $\gamma'\subset G_X$ satisfying
\[
 \ell_{G_X}(\gamma')
 \leq \cosh(w(X))\,\ell_X(\gamma) < \frac{4\pi q}{\sys(X)}\, \ell_X(\gamma) 
\]
\end{lemma}

Here are some basic facts about $G_X$ and a brief idea of how you prove the above statement. The geometric spine $G_X$ is defined as the cut-locus of the distance function to $\partial X$, so edges correspond to points with exactly two distance paths to boundary and vertices to points with at least three such paths. The graph, together with the distance paths from vertices to the boundary, decompose the surface into quadrilaterals. The proof consists in cutting a curve into geodesic segments that cross the quadrilaterals and by pushing the segments to the graph in the natural way. The quantified distortion factor comes from standard hyperbolic geometric arguments. Finally, we note that the girth of $G_X$ is at least the systole of $X$, which we denote $s_0= \sys(X)$, and which we recall is the minimal length among all closed geodesics of $X$ including its boundary curves. 

\begin{figure}[h]
 \centering
\begin{overpic}[width=0.6\linewidth,
tics=10]{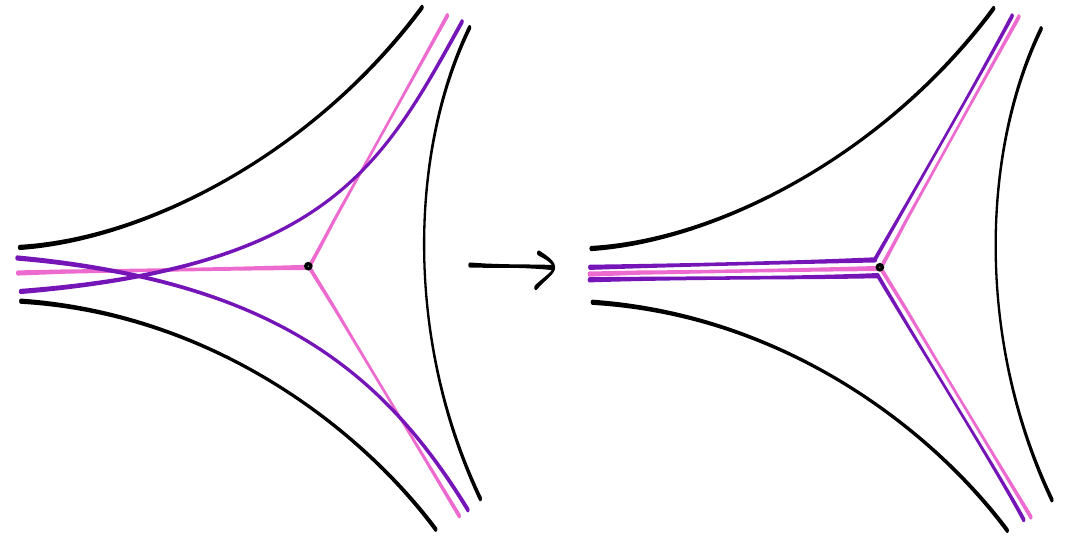}
\put(33,20){$G_X$}
\put(19,11){$\partial X$}
\end{overpic}
\caption{Pushing a curve to a carrier-graph}
 \label{fig:graph}
\end{figure}

For a generic surface, $G_X$ has all vertices of degree $3$, but can have higher degree vertices. If so, for convenience, we split a vertex of higher degree into several of degree $3$, and add edges of length $0$. There are multiple ways of doing this: in fact, if the vertex is of degree $d$, there are Catalan number $d-2$ ways of doing this, but we just choose one. (This is not strictly necessary of course, but just makes our life a bit easier.)

Now since the vertices are degree $3$, an Euler characteristic argument tells us that the number of edges is $3 |\chi(X)|= 3q$, and the number of vertices is $2q$. 

Our goal is reach a purely combinatorial setting, and $G_X$ is an intermediate step. To go one step further, we define $G$ to be the graph associated to $G_X$ where we have replaced all edges of $G_X$ by edges of length $1$.

The following lemma compares lengths of closed paths on $G$ and $G_X$. As these are the same underlying abstract graphs, a closed path that lives on $G_X$ naturally lies on $G$ and vice-versa. For a closed path $\gamma$, its combinatorial length is its $G$-length, denoted by $\ell_G(\gamma)$, while its geometric length is $\ell_{G_X}(X)$. 

\begin{lemma}\label{lem:comb}
Let $\gamma$ be a closed path on $G_X$ with $\ell_{G_X}(\gamma) \leq L$. Its combinatorial length satisfies 
\[
\ell_G(\gamma) \leq \frac{2q}{s_0} L + 2q.
\]
\end{lemma}
\begin{proof}
Let $\gamma$ be a cyclically reduced closed path in $G_X$. Let $L'$ be its $G$-length. Now since the number of vertices of $G$ is $2q$, and any consecutive sequence of $2 q$ edges of $\gamma$ visits $2q+1$ vertices, at least one of them repeats. The corresponding closed path has $G_X$ length at least $s_0$. Thus any block of $2q$ edges of $\gamma$ has $G_X$ length at least $s_0$. It follows that we write
\[
L'= k (2q) + r
\]
with $0\leq r< 2q$ then 
\[
L\geq \ell_{G_X}(\gamma) \geq k s
\]
and so 
\[
\ell_G(\gamma) = L' \leq 2q \lfloor \frac{L}{s_0} \rfloor + 2q -1 < 2q \left(\frac{L}{s_0} +1\right).
\]
\end{proof}

To sum up, these two lemmas show that any closed geodesic of length at most $L$ on $X$ has a corresponding closed path on $G$ of length at most combinatorial length $L_c$ where 
\begin{equation}\label{eq:comb}
L_c \leq \frac{ 2 q \cosh(w(X))}{s_0} L + 2q.
\end{equation}
on $G$. Similarly, if instead of a closed geodesic $\gamma$ we consider a primitive geodesic multicurve $\Gamma$, the same result holds. 

We now analyze the combinatorial setting. Before we start to count, it is worth observing a few things. One thing is that a closed path in a graph gives us a weighted list of edges it passes through. Dual to $G$, which is trivalent, we have a hexagon decomposition. And for the same reason that normal coordinates\footnote{Normal coordinates are intersection numbers with a triangulation.} of a curve determine a curve, this weighted list of edges determines a unique simple multicurve. And since the number of simple multicurves grows polynomially in terms of their length, this information is quite far from determining the original multircurve. We get slightly more information by looking at the number of weighted oriented edges, but this is still a polynomial amount of information.

We won't appeal to anything fancy for this. Since we have $3q$ edges on $G$, it has $6q$ oriented edges. An important observation is that a non-backtracking path on $G$ that passes through an oriented edge has two possible predecessors and two possible successors, as illustrated in Figure \ref{fig:oriented}.

\begin{figure}[h]
 \centering
\begin{overpic}[width=0.6\linewidth,
tics=10]{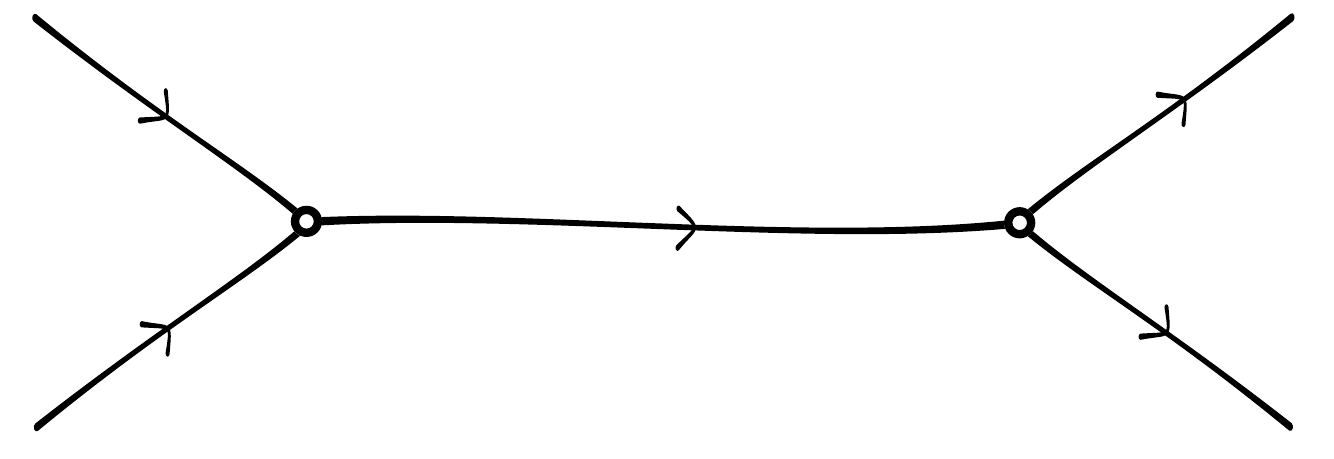}
\end{overpic}
\caption{An oriented edge of $G$ with its possible predecessors and successors}
 \label{fig:oriented}
\end{figure}

So for any integer $L_c>0$, which should be thought of as an upper bound on the combinatorial length of a multicurve, there are at most 
\begin{equation*}\label{eq:orient}
\binom{L_c+6q-1}{6q-1}
\end{equation*}
ways to choose partition its $L_c$ segments into $6q$ parts. But we will need to go one step further. 

We denote by $E_G$ the set of oriented edges of $G$. Consider $e\in E_G$. It has two predecessors, namely the two oriented edges pointing towards it. In particular, if a multicurve passes through $e$, it arrives from one of these two predecessors. For reasons that will become apparent in what follows, we will need to split the number of oriented segments parallel to a given oriented edge into two parts, depending on their predecessor. This amounts to splitting the segments into $12 q$ sets: two for each oriented edge. We bound the full number of choices by 
\begin{equation}\label{eq:poly}
\binom{L_c+12q-1}{12q-1} \leq \left( \frac{e}{12q-1} \right)^{12q} \left( L_c+12q -1\right) ^{12q}.
\end{equation}

To recap, this number of choices determines a local picture for a multicurve: 
\begin{itemize}
\item We know the number of its segments corresponding to each oriented edge of $G$.
\item If a multicurve passes through $e\in E_G$ exactly $m=m_0+m_1$ times, we know how many come from the left ($m_0$) and how many come from the right ($m_1$).
\end{itemize}

This is an obvious overcount: many choices don't correspond to any multicurves at all. For instance, we haven't taken into amount the fact that the number of incoming segments must be the same as the number of outgoing segments, but as we are only aiming for an upper bound, and as these numbers are only polynomial in any event, we leave it at that. 

\subsection{Words}
Now comes the most delicate part of the encoding. Throughout, we consider a fixed multicurve $\Gamma$. 

{\it Ordering segments.} To each oriented edge $e\in E_G$, we look at the segments of $\Gamma$ which pass through $e$ and order them from left to right as follows. The universal cover $T$ of $G$ is a trivalent tree. We fix $\tilde{e}$, a lift of $e$. This splits the limit points of $T$, which can be thought of a subset of $\mathcal{S}^1$, into two parts, the forward limit points and the backward limit points. In particular the forward limit points can be thought of a subset of a finite segment. Now each segment of $\Gamma$ belongs to a component of $\Gamma$ which lifts to an infinite path on $T$. Any two distinct segments of $\Gamma$ correspond to distinct forward limit points. From this, we can deduce a corresponding left to right order: a segment $s$ is left of $s'$ if their corresponding limit points have the same order (see Figure \ref{fig:order}). 

\begin{figure}[h]
 \centering
\begin{overpic}[width=0.5\linewidth,
tics=10]{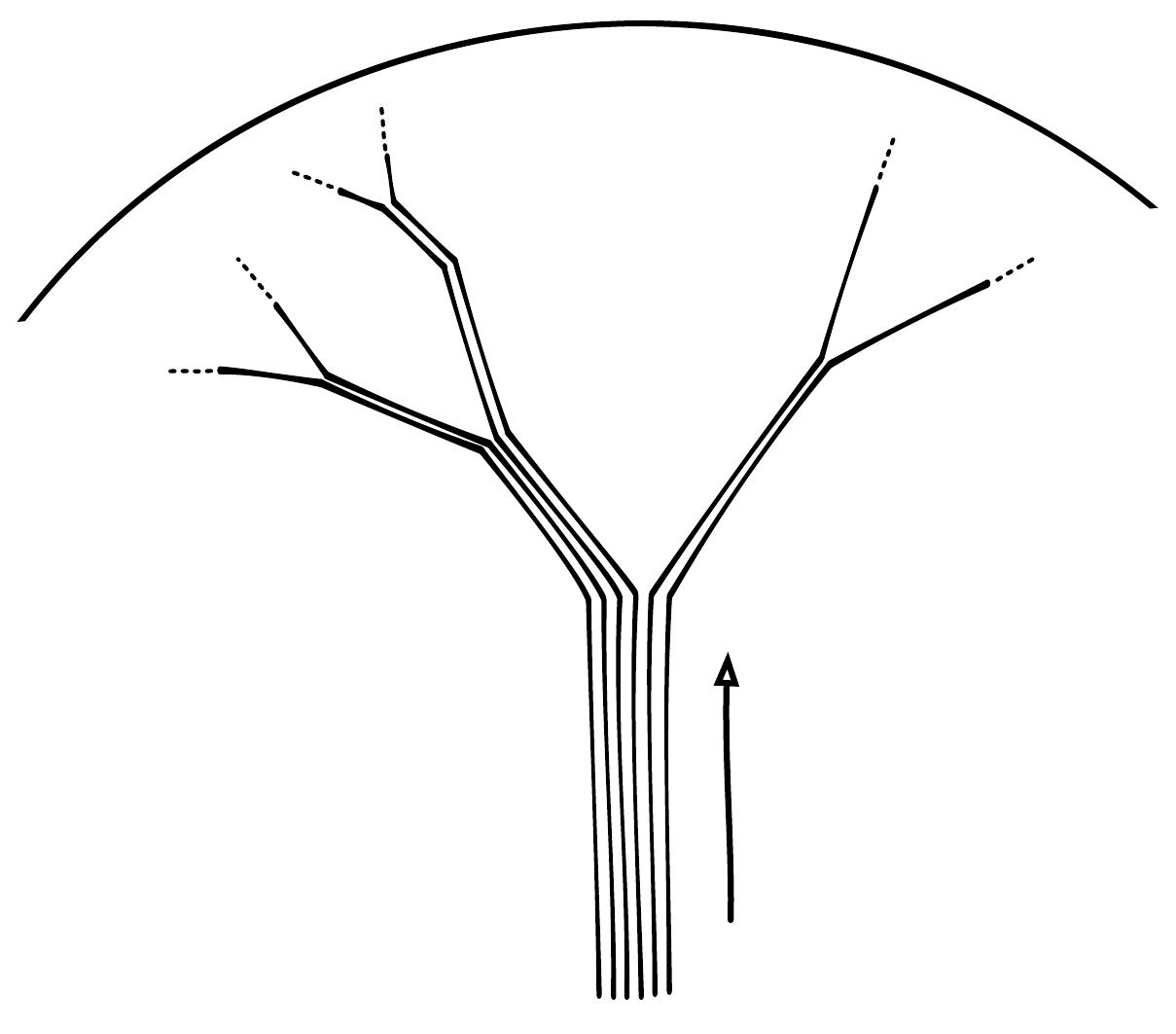}
\put(12,78){$\partial \Hyp$}
\end{overpic}
\caption{Ordering oriented segments by forward limit points (order on the figure is from left to right because, in English, we read words from left to right)}
 \label{fig:order}
\end{figure}

This makes sense: in the universal cover, the two corresponding simple rays must split at one point. 

Now for each $e\in E_G$, we can now associate a word $w_e$ of $0$s and $1$s as follows. We consider the ordered segments and associate a $0$ to it if its predecessor was to its left, and a $1$ if otherwise (see Figure \ref{fig:words}). For a given primitive multicurve $\Gamma \in \multi(X)$, we denote by $w_e(\Gamma)$ the associated word to oriented edge $e\in E_G$.

\begin{figure}[h]
 \centering
\begin{overpic}[width=0.4\linewidth,
tics=10]{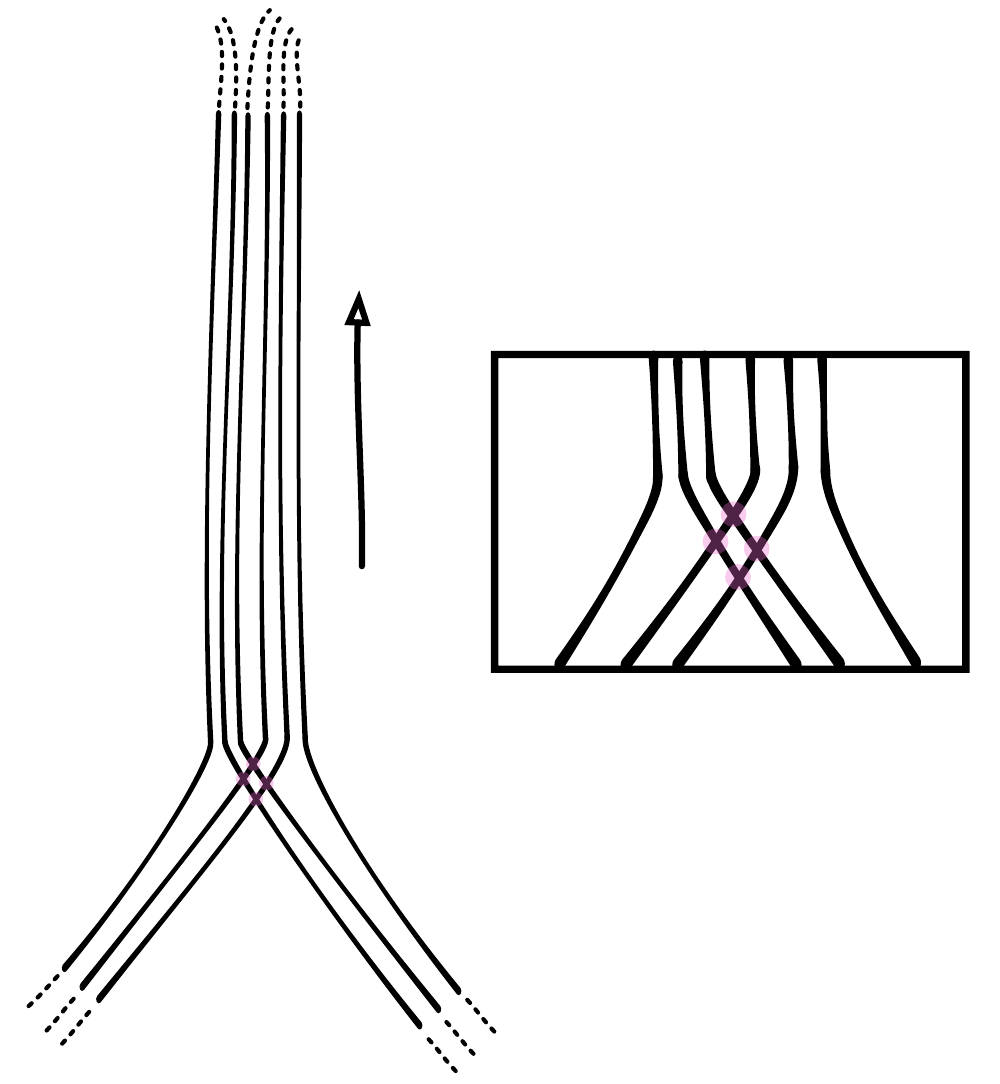}
\put(8,0){$0$}
\put(36,0){$1$}
\end{overpic}
\caption{The local picture associated to the word $011001$}
 \label{fig:words}
\end{figure}

We will need two essential facts about the collection of words: they determine a unique multicurve and we can read a lower bound on the intersection number of the multicurve from the words. 

A set of words is called admissible if it can be realized by some primitive multicurve. We now show that it uniquely determines the multicurve as well. 

\begin{lemma}\label{lem:wordswork} An admissible set of words $w_e$, for $e\in E_G$, determines a unique primitive multicurve. 
\end{lemma}

\begin{proof}
A multicurve is uniquely determined by a collection of cyclic collections of edges. The proof essentially boils down to the fact that you can "read words backwards", and recover the cyclic collections of edges. We provide a more formal proof. 

For $e,f\in E_G$, consider a possible transition $ef$ (that is, $e$ is a possible predecessor of $f$). The positions in $w_f$ whose predecessor is $e$ form a known ordered subset
$C_{e,f}$. Among the positions of $w_e$, those with successor $f$ form a consecutive block $B_{e,f}$, the other consecutive block corresponding to the only other possible successor. As the set of words are admissible, its size is $|C_{e,f}|$, so both blocks are known. We can now match $B_{e,f}$ to $C_{e,f}$ in order. This is the actual successor map. In fact, the deck transformation taking the successor lift of $f$ to its fixed lift is the same for the whole block and preserves the cyclic order of ends. These matchings reconstruct a permutation of all word positions. Its cycles recover exactly the component cyclic words. In particular, no extra choice of a decomposition into components is required.
\end{proof}

Reading a word from left to right, we are interested in pairs of $1$s followed by $0$. Indeed, any such pair corresponds to a pair of axes with alternating endpoints. We call such a pair an {inversion}, and for a word $w$ of $0$s and $1$, we denote by $\inv(w)$ its number of inversions. For instance $\inv(0100110)=5$. 

\begin{lemma}\label{lem:inversions}
Let $\Gamma \in \multi(X)$ and $w_e(\Gamma), e\in E_G$ its set of words. 
Then
\[
\sum_{e\in E_G} \inv(w_e(\Gamma)) \leq \frac{1}{2} i(\Gamma,\Gamma).
\]
\end{lemma}

\begin{proof}
The proof uses some of the general ideas in a paper of Birman and Series to compute intersection numbers \cite{BirmanSeries84}, originally between simple curves, but extended to computing more general intersection numbers in \cite{DespreLazarus19}.

An inversion in $w_e$ determines two periodic lines whose four ends alternate. The lifted spine is an equivariantly embedded quasi-isometric tree in the convex universal cover of $X$. Its identification of ends with the limit set preserves cyclic order. The corresponding hyperbolic axes therefore cross. The lines enter $\widetilde e$ through different predecessor branches and traverse it in the same direction. Hence $\widetilde e$ is the first edge of their common segment. At every subsequent common edge their predecessors agree, and thus this pair produces at most one inversion. Distinct periodic lines have only a finite common connected portion since a common infinite ray would force identical periodic futures.

This uniqueness persists modulo deck transformations. A transformation identifying two such unordered pairs preserves their chosen orientations and their first common oriented edge. Both first edges are the chosen lift of the same projected edge, so the transformation fixes an edge and is the identity. Here each lift has the orientation pulled back from its component; deck transformations preserve these orientations. No lift belongs to two distinct components, because their unoriented primitive classes are distinct. Thus inversions inject into deck orbits of crossing pairs of axes. Such an orbit records one self-crossing of a component or one intersection of two different components.
\end{proof}
\subsection{Counting words}
We now give an upper bound on the number of admissible sets of words, which will depend on $K$. 

We suppose that for each $e\in E_G$, we know the number of $0$s and $1$s in $w_e$. What we don't yet know is the exact word. 

For a word $w_e$, let $m_0(e)$ be its number of $0$s and $m_1(e)$ its number of $1$s. To each $1$ in $w_e$, we associate the number of $0$s to its right (equivalently, the number of inversions it determines). This produces a sequence $n_1 \geq \hdots \geq n_{m_1(e)}$ of non-negative integers. 

We make the following observation. 
\begin{proposition}
The integers $n_1\geq \hdots \geq n_{m_1(e)}$ determine $w_e$ exactly.
\end{proposition}
\begin{proof}
The word is constructed from the right as follows: place next to the rightmost $1$ its number of corresponding $0$s (that is $n_{m_1(e)}$). Now iterate the process by placing $n_{m_1(e)-1} - n_{m_1(e)}$ next to the penultimate $1$ and so forth. (Equivalently, place exactly $n_{k} - n_{k+1}$ $0$s next to the $k$-th $1$ read from the left, for $k= 1,\hdots, m_1(e)-1$.) Now, to finish, place all remaining $0$s to the left of the leftmost $1$.

Note that this only works because we suppose that we already know the total number of $0$s of $w_e$. 
\end{proof}

Now, quite conveniently, the sum $\sum_{k=1}^{m_1(e)} n_k$ is exactly the number of inversions of the word $w_e$. In particular, we can count how many inversions we can possibly create with at most $K$ self-intersections. Let $K_1,\hdots,K_{6q}$ be non-negative integers each corresponding to the number of inversions of an different oriented edge of $G$. They must satisfy
\[
\sum_{i=1}{6q} K_i=K. 
\]
By the above proposition, it suffices to count, for each choice of $K_1,\hdots,K_{6q}$, the number of possible partitions of each $K_i$. All in all, the number of words is bounded above by the number of $6q$-tuples of partitions of total sum at most $K$. For this, we use the following estimate.

\begin{lemma}\label{lem:partition}
The number $P_{6q}(K)$ of $6q$-tuples of partitions of total sum at most $K$ is bounded above by 
\[
 P_{6q}(K)\le e^{2\pi\sqrt{qK}}
  \]
Hence, the number of possible choices of admissible words is also bounded above by the same quantity.
\end{lemma}
\begin{proof}
This is a standard counting estimate. The partition product, see e.g. \cite{Andrews}, gives, for every $a>0$,
\[
 P_{6q}(K)\le e^{aK}\prod_{n\ge1}(1-e^{-an})^{-6q}
 \]
 and thus  
 \[
 \log P_{6q}(K)\le aK+6q\sum_{r\ge1}\frac1{r(e^{ar}-1)}
 \le aK+\frac{q\pi^2}{a}.
\]
We can now optimize with respect to $a$. Setting $a= \pi \sqrt{\frac{q}{2K}}$ gives the result.
\end{proof}

Note that, for $K=0$, it equals $1$ (which is reassuring). 

Put all together, we can prove the following restatement of Theorem \ref{thm:boundaryintro} from the introduction. 

\begin{theorem}\label{thm:boundary}
Let $X$ be a hyperbolic surface with geodesic boundary with $|\chi(X)| = q$ and $s_0= \sys(X)$. Then, for any $L>0$, we have
\[
N_X(L,K)  \leq 
\left(\frac{e^{4\pi q/s_0}}{2s_0}\,L+4\right)^{12q} e^{2\pi\sqrt{qK}}.
\]
\end{theorem}

\begin{proof}
To prove the result, it suffices to compile the previously obtained estimates, namely Lemmas \ref{lem:spine}, \ref{lem:comb}, the estimate in equation \ref{eq:poly}, and Lemma \ref{lem:partition}. By combining these inequalities, we have
\begin{eqnarray*}
N_X(L,K) & \leq &
\left(\frac{e}{12q-1}\right)^{12q}
\left(2q\left(\frac{e^{\frac{4\pi q}{s_0}}}{s_0}L+1\right)+12q-1\right)^{12q}
\\
& = & \left(
\frac{2q\,e^{1+4\pi q/s_0}}{(12q-1)s_0}\,L
+\frac{e(14q-1)}{12q-1}
\right)^{12q}\\
& < & \left(\frac{e^{4\pi q/s_0}}{2s_0}\,L+4\right)^{12q} e^{2\pi\sqrt{qK}}.
\end{eqnarray*}

\end{proof}

\section{Closed surfaces}\label{sec:closed}

We now focus on the setting of closed surfaces.

The strategy is as follows: given a closed surface $X$, we choose an embedded metric disk and remove a smaller concentric disk to obtain a surface with a hole $X_0$. Lengths of curves on $X_0$ are comparable to lengths of curves on $X$ (this is illustrated in Figure \ref{fig:hole}). 

\begin{figure}[h]
 \centering
\begin{overpic}[width=0.6\linewidth,
tics=10]{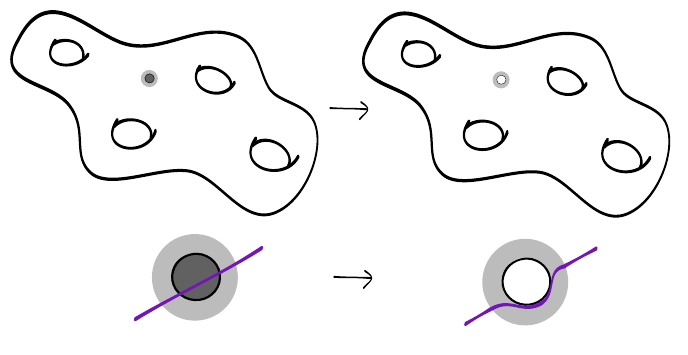}
\end{overpic}
\caption{Punching a hole into $X$ to obtain $X_0$}
 \label{fig:hole}
\end{figure}

We then construct a hyperbolic surface $X_1$ with geodesic boundary, homeomorphic to $X_0$, on which lengths of curves are controlled by those on $X_0$, and hence by those on $X$. Filling the boundary does not induce a bijection of free homotopy classes, but all we need is an injective map from closed geodesics of $X$ to those of $X_1$ that controls lengths and does not increase self-intersection numbers. We can then apply the results of the previous section.

By a theorem of Yamada \cite{Yamada82}, $X$ contains an embedded disk
of radius
\[
r_1=\arcsinh\!\left(\frac{2}{\sqrt{3}}\right).
\]
Choose $p$ to be the centre of such a disk, set $r_0=r_1/2$, and remove the open disk of radius $r_0$ centered at $p$ to obtain $X_0$. 

To study $X_0$, we begin with a more general geometric lemma comparing distances in a disk with intrinsic distances after removing a smaller concentric disk.

\begin{lemma}\label{lem:disk-detour}
Let $r_1=\arcsinh(2/\sqrt{3})$ and $r_0=r_1/2$. For any $p,q\in\partial B_{r_1}$,
\[
d_{B_{r_1}\setminus B_{r_0}}(p,q)
\leq \frac65\,d_{B_{r_1}}(p,q),
\]
where the distance on the left is intrinsic and paths along the inner boundary are allowed.
\end{lemma}

\begin{proof}
The strategy is to compare distances between points on the boundary of the disk without a hole punched into it and the disk with a hole. This is illustrated in Figure \ref{fig:estimate}. 

\begin{figure}[h]
 \centering
\begin{overpic}[width=0.4\linewidth,
tics=10]{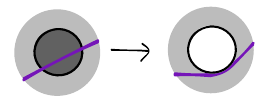}
\put(3.5,5){$p$}
\put(38,23){$q$}
\end{overpic}
\caption{A sketch of how to compare distances on the two disks.}
 \label{fig:estimate}
\end{figure}

Let $\theta\in[0,\pi]$ be the smaller central angle between $p$
and $q$. Using the geometry of the triangle they form with the center of the disk, their distance in the disk without the hole is
\[
d_D(p,q) =2\arcsinh\!\left(\sinh r_1\sin(\theta/2)\right) =:f(\theta)
\]
which we set to be a function of the angle $f(\theta)$ as $r_1$ is fixed.

\begin{figure}[h]
 \centering
\begin{overpic}[width=0.6\linewidth,
tics=10]{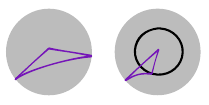}
\put(22.5,22){$\theta$}
\put(71.5,20){$\alpha$}
\put(11,21){$r_1$}
\put(75,17){$r_0$}
\put(66,10){$\delta$}
\end{overpic}
\caption{How angles and distances relate.}
 \label{fig:estimate2}
\end{figure}

We again make use of the hyperbolic trigonometry of a (right-angled) triangle. Figure \ref{fig:estimate2} portrays the triangles used in this argument. A tangent segment from an outer endpoint to the inner circle has length $\delta$ where \[
\delta=\arccosh\!\left(\frac{\cosh r_1}{\cosh r_0}\right).\]
And its central angle $\alpha$ satisfies
\[
\alpha=\arccos\!\left(\frac{\tanh r_0}{\tanh r_1}\right).
\]

If $\theta\leq2\alpha$, the geodesic chord avoids the open inner
disk, so there is nothing to prove. Otherwise, the two tangent
segments and the intervening inner-circle arc give an admissible
path of length
\[
l(\theta)=2\delta+(\theta-2\alpha)\sinh r_0.
\]
At $\theta=2\alpha$, we have $l(\theta)=f(\theta)$ and $l'=f'=\sinh r_0$.
Since $f$ is a concave function and $l$ is affine,
\[
(l'f-lf')'=-lf''\geq0.
\]
Thus $\sfrac{l}{f}$ is nondecreasing on $[2\alpha,\pi]$. Its maximum is
attained at $\theta=\pi$, and substitution gives
\[
\frac{l(\pi)}{f(\pi)}
=
\frac{2d+(\pi-2\alpha)\sinh r_0}{2r_1}
<\frac65.
\]
This proves the claim.
\end{proof}

We next choose a triangulation of a closed surface $X$ with controlled angles. To construct the triangulation, we choose any point $p\in X$ and consider its Vorono\"i decomposition around $p$. We now choose a triangulation dual to this decomposition. As we argued before, for a Vorono\"i decomposition around a boundary curve, if the Vorono\"i graph is trivalent, its dual is exactly the triangulation we want. If not, we triangulate the larger cells in any way, using diagonals. The obtained triangulation is a Delaunay decomposition, in that it lifts to a genuine Delaunay triangulation of the the universal cover $\Hyp$ (see \cite{DeBlois18, DespreSchlenkerTeillaud24}). 

\begin{lemma}\label{lem:delaunay-angles}
Let $X$ be a closed hyperbolic surface of genus $g\geq2$ with $\sys(X)=s_0$. For any $p\in X$, the minimum angle in triangles of the Delaunay decomposition based at $p$ satisfies
\[
\theta_0\geq 2s_0\, e^{-\frac{32(g-1)}{s_0}}.
\]
\end{lemma}

\begin{proof}
Take a Delaunary triangulation constructed as above, that is with arcs either dual to the Voronoi decomposition or by adding additional diaogonals. Note that each resulting triangle has the same circumcircle as its original face.

Let $R$ be the radius of such a circle. Its center projects to a Voronoi vertex $v$, and $R=d(p,v)$. Along a distance path from $p$ to $v$, choose points at distances $ks_0$, for $k=0,\ldots,\lfloor R/s_0\rfloor$. Their open disks of radius $s_0/2$ are embedded and pairwise disjoint. We can now compare areas to get
\[
\left(\left\lfloor\frac{R}{s_0}\right\rfloor+1\right)
2\pi\bigl(\cosh(s_0/2)-1\bigr)
\leq4\pi(g-1).
\]
Consequently,
\[
R<
\frac{2(g-1)s_0}{\cosh(s_0/2)-1}
\leq\frac{16(g-1)}{s_0}.
\]

Now consider a lifted triangle with side lengths $a,b,c$ and angle $\theta$ opposite $a$. Its vertices are distinct lifts of $p$, so $a\geq s_0$, and $b,c\leq2R$. By hyperbolic trigonometry in the triangle we have
\[
\sin\theta
=
\frac{\sinh(a/2)}
{\tanh R\,\cosh(b/2)\cosh(c/2)}
\geq
\frac{\sinh(s_0/2)}{\sinh R\cosh R}
\geq2s_0 \,e^{-2R}.
\]
Using $\theta\geq\sin\theta$ and the bound on $R$ proves the claim.
\end{proof}

We now prove a lemma which allows us to compare lengths of arcs on a truncated triangle and a hexagon. The triangle is truncated by all points distance at most $r$ from a vertex. The remaining side lengths, $\ell_1,\ell_2,\ell_3$, are supposed to be positive. The comparison hexagon is the corresponding right-angled hexagon with alternating side lengths $\ell_1,\ell_2,\ell_3$ (see Figure \ref{fig:TriHex}). 

\begin{figure}[h]
 \centering
\begin{overpic}[width=0.6\linewidth,
tics=10]{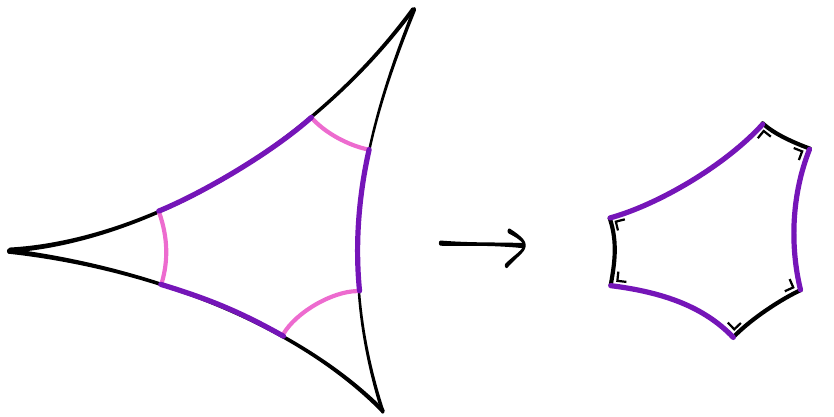}
\put(11,23.5){$r$}
\put(11,14){$r$}
\put(39,41){$r$}
\put(37,03){$r$}
\put(47,35){$r$}
\put(46,08){$r$}
\put(44,21){$\ell_1$}
\put(24,32){$\ell_2$}
\put(24,09){$\ell_3$}
\put(97,21){$\ell_1$}
\put(80,31){$\ell_2$}
\put(78,10){$\ell_3$}
\end{overpic}
\caption{Comparing $T$ and $H$}
 \label{fig:TriHex}
\end{figure}

\begin{lemma}\label{lem:truncate}
Let $T$ be a hyperbolic triangle with sides of lengths $a_i=2r+\ell_i$, for $i=1,2,3$, where $r>0$ and
$\ell_i\geq\ell_0>0$, and suppose its angles are at least $\theta_0>0$. Let $H$ be the right-angled hyperbolic hexagon with alternating side lengths $\ell_1,\ell_2,\ell_3$.

We identify the middle segment $t_i$ of length $\ell_i$ on each side of $T$ with the corresponding side $h_i$ of $H$ by arclength, and denote this correspondence by $\varphi$. Then, for all $p,q\in t_1\cup t_2\cup t_3$, we have
\[
d_H(\varphi(p),\varphi(q))
\leq
\frac{2r+\coth\ell_0}{r\sin(\theta_0/2)}
\,d_T(p,q).
\]
\end{lemma}

\begin{proof}
If $p,q \in t_1\cup t_2\cup t_3$ lie on the same side, their distance is preserved and there is nothing to prove. So suppose not, and let $v$ be the common vertex of their triangle sides, and let $\theta\geq\theta_0$ be the angle at $v$. We can express the distances between the vertices are our points as
\[
d_T(v,p)=r+x,\qquad d_T(v,q)=r+y,
\qquad x,y\geq0.
\]
The triangle cosine rule gives
\begin{align*}
\cosh d_T(p,q)
&=
\sin^2(\theta/2)\cosh(2r+x+y)
+\cos^2(\theta/2)\cosh(x-y)\\
&\geq
1+2\sin^2(\theta_0/2)
       \sinh^2\!\left(r+\frac{x+y}{2}\right)\\
&\geq
\cosh\!\left(\sin(\theta_0/2)(2r+x+y)\right).
\end{align*}
The last inequality follows from
$\sinh(bt)\leq b\sinh t$ for $0\leq b\leq1$.
Thus
\[
d_T(p,q)\geq\sin(\theta_0/2)(2r+x+y).
\]

Let $e$ be side of $H$ between the sides containing $\varphi(p)$ and $\varphi(q)$. The hexagon cosine rule and the triangle inequality $\ell_i<\ell_j+\ell_k+2r$ give
\[
\cosh\ell(e)
=
\frac{\cosh\ell_i+\cosh\ell_j\cosh\ell_k}
{\sinh\ell_j\sinh\ell_k}
\leq
e^{2r}+(1+e^{2r})\coth^2\ell_0.
\]
Now using the inequalities $\arccosh(u) \leq\log(2u)$ and $\log(2+4u^2)\leq 2u$ for $u\geq1$, it follows that
\[
\ell(e)
\leq 2r+\log\!\left(2+4\coth^2\ell_0\right)
\leq 2r+2\coth\ell_0.
\]
We now bound the distance between $\varphi(p)$ and $\varphi(q)$ in the most obvious way, by considering the concatenated path along the boundary to obtain
\[
d_H(\varphi(p),\varphi(q))
\leq x+y+2r+2\coth\ell_0
\leq
\frac{2r+\coth\ell_0}{r\sin(\theta_0/2)}
\,d_T(p,q),
\]
as required.
\end{proof}

We can now construct the surface $X_1$ with one boundary component and transfer curves to it. 

Before doing so, we make an important observation. There is a well-defined map topological map from (interior) curves on $X_1$ to curves on $X$ defined by filling the boundary component by a disk: 
\[
\Phi: \curves(X_1)\to \curves(X).
\]
This map $\Phi$ is the same as what is generally called a "forgetful" map between a punctured surface and a surface where a puncture is "forgotten". In the case of an actual puncture, If this map is done conformally, meaning that the closed surface is conformally equivalent to the punctured surface with a punctured filled in, then by the Schwartz lemma, curves decrease in length. Here however, we need to construct $X_1$ to be begin with, and once we do so, we are comparing a surface with geodesic boundary to a closed surface, so there is no obvious natural map as in the case with punctures. Everything needs to be proved by hand. Furthermore, since we are starting with a closed surface, we need to compare a smaller set of curves to curves on the surface with boundary, which requires constructing a type of inverse to $\Phi$: the map $\rho$ we will construct has the property that each curve on $X$ will be sent to a curve which lies in the preimage of $\Phi$. 

We now prove two lemmas: in the first lemma we give the construction and prove the more delicate part, that is that the lengths of curves under our map $\rho$ are at most increased by a factor we can control. The second lemma is more direct, and shows that lengths on $X_1$ are least those of $X$, giving us the needed control on the systole of $X_1$. 

\begin{lemma}\label{lem:injection}
Let $X$ be a closed hyperbolic surface of genus $g\geq2$ with $\sys(X)=s_0$. There exist a compact hyperbolic surface $X_1$ of genus $g$ with one geodesic boundary component and an injective map
\[
\rho:\curves(X)\longrightarrow\curves(X_1)
\]
such that, for every $\gamma\in\curves(X)$,
\[
\ell_{X_1}(\rho(\gamma))
\leq
\frac{10}{s_0}e^{\frac{32(g-1)}{s_0}}\ell_X(\gamma)
\]
Moreover, $i(\rho(\gamma),\rho(\gamma))
= i(\gamma,\gamma)$ and $\Phi(\rho(\gamma)) = \gamma$. 
\end{lemma}

\begin{proof}
Choose the embedded disk centered at $p$ as above, with $r_1=\arcsinh(2/\sqrt3)$ and $r_0=r_1/2$. Let $\mathcal T$ be the one-vertex triangulation based at $p$ given by Lemma~\ref{lem:delaunay-angles}.

Note that every edge of $\mathcal T$ has length at least $2r_1$, because it enters and exits the disk. Removing the disk of radius $r_0$ results in truncated triangles with the truncated sides of length at least $r_1$.

We now replacece each triangle by the right-angled hexagon whose alternating sides have these remaining lengths, and glue corresponding sides according to $\mathcal T$.
The resulting surface $X_1$ has genus $g$ and one boundary component which, due to the pasting of geodesics at right angles, results in a geodesic boundary component.

Let $\gamma$ be a closed geodesic on $X$. First replace its passages through $B(p,r_1)$ by detours around $B(p,r_0)$. We apply Lemma~\ref{lem:disk-detour} to each of the concerned arcs. 

We make an important observation: these replacements can be made simultaneously without increasing the number of intersections. One way to think of this is to imagine the point $p$ as disjoint from the original curve on $X$, and letting a disk grow around $p$. At first, the disk is disjoint and then as it grows, this forces the curve to circumvent the smaller disk. Note that the curve remains unchanged outside of the disk of radius $r_1$. During the process of growing the disk, no transversal intersection points are added. More formally, we can use the so-called bigon criterion (see \cite[Section~1.2]{FarbMargalit12} or \cite{HassScott85}): shortest representatives of the passage arcs in their annular homotopy classes are in minimal position and a bigon would allow a shortening. Note that we don't need that the intersection isn't reduced by this process (but it isn't). 

Now, according to Lemma~\ref{lem:disk-detour}, the ratio of the lengths of the modified arcs can be bounded by $\frac{6}{5}$ and hence the resulting curve $\gamma_0$ on $X_0$ satisfies
\[
\ell_{X_0}(\gamma_0)\leq\frac65\,\ell_X(\gamma)
\]
and 
\[
i(\gamma_0,\gamma_0) = i(\gamma,\gamma).
\]
where $\gamma_0$ is the resulting curve in $X_0$.

Now we cut $\gamma_0$ at its intersection points with $\mathcal T$. Every intersection point lies on a retained middle segment, since $\gamma_0$ stays outside $B(p,r_0)$. We can transfer the endpoints by the arclength identifications and replace each triangle passage by the geodesic segment joining its corresponding endpoints in the hexagon. The segments match across glued sides and form a closed curve on $X_1$.

Again, this operation does not increase intersections. In each hexagon, two geodesic segments cross only when their endpoints alternate on its boundary and this is the same on the (truncated) triangles. The endpoint identifications preserve their cyclic order.

Let $\rho(\gamma)$ be the closed geodesic in the resulting free homotopy class. Straightening decreases length and does not increase geometric self-intersection number (but again, it doesn't increase intersection number either in this case). We apply Lemma~\ref{lem:truncate}, with $r=r_0$ and $\ell_0=r_1$ to obtain
\[
\ell_{X_1}(\rho(\gamma))
\leq
\frac65\,
\frac{2r_0+\coth r_1}{r_0\sin(\theta_0/2)}
\,\ell_X(\gamma).
\]
Since $\sin(\theta_0/2)\geq\theta_0/\pi$ and Lemma~\ref{lem:delaunay-angles} gives $\theta_0\geq2s_0e^{-32(g-1)/s_0}$, we obtain
\[
\ell_{X_1}(\rho(\gamma))
\leq
\frac{6\pi}{5}
\left(1+\frac{\coth r_1}{r_1}\right)
\frac{e^{32(g-1)/s_0}}{s_0}\,\ell_X(\gamma)
<
\frac{10}{s_0}e^{32(g-1)/s_0}\ell_X(\gamma).
\]
The intersection estimate follows from the construction.

Finally, we show that $\Phi(\rho(\gamma))=\gamma$. This amounts to showing that filling the boundary reverses the hole punching operation. Each transferred passage then has the same endpoints and homotopy class in its triangle as the original passage, so the resulting closed curve is freely homotopic to $\gamma$. This proves injectivity. It also shows that $\rho(\gamma)$ is neither trivial nor peripheral. If it were a proper power, filling would make $\gamma$ a proper power as well. Thus $\rho(\gamma)$ is primitive.
\end{proof}

We now pass to the second lemma, which shows that the surface $X_1$ has additional important properties. Its interior systole is at least the systole of $X$, and its boundary geodesic is of length at least a constant.

\begin{lemma}\label{lem:systole}
The surface $X_1$ satisfies
\[
\ell_{X_1}(\gamma) \geq \ell_X(\Phi(\gamma))\
\]
for all $\gamma \in \curves(X_1)$ which survive under the forgetful map. 
Moreover,
\[
\ell_{X_1}(\partial X_1)\ge 2\pi\sinh r_0
\qquad\text{and}\qquad
\sys(X_1)\geq \min\{s_0,2\pi\sinh r_0\}
\]
where $s_0 = \sys(X)$. 
\end{lemma}

\begin{proof}
Consider a triangle of $\mathcal T$ whose side lengths are $a+2r_0$, $b+2r+0$ and $c+2r_0$, so that the corresponding hexagon has alternating side lengths $a$, $b$ and $c$. Let $\theta$ be the angle between the original sides of lengths $a+2r_0$ and $b+2r_0$, and let $h$ be the length of the boundary side of the hexagon joining its sides of lengths $a$ and $b$.

The hyperbolic cosine law in the triangle gives 
\[
1-\cos\theta
=
\frac{\cosh(c+2r_0)-\cosh(a-b)}{\sinh(a+2r_0)\sinh(b+2r_0)}
\]
whereas the same formula for the hexagon gives
\[   
\cosh h-1 =
\frac{\cosh c+\cosh(a-b)} {\sinh a\,\sinh b}.
\]
Using
\[
\sinh(x+2r_0)\ge e^{2r_0}\sinh x, \qquad
\cosh(x+2r_0)\le e^{2r_0}\cosh x
\]
for $x\geq 0$ we obtain
\begin{equation}\label{eq:trunc}
\cosh h-1\ge e^{2r_0}(1-\cos\theta).
\end{equation}

Now take points $p,q$ on the hexagon sides of lengths $a$ and $b$, at distances $u$ and $v$ from their respective endpoints on the side of length $h$. The corresponding points ($\varphi(p),\varphi(q)$ using the notation from Lemma \ref{lem:truncate}) in the original triangle lie at distances $r+u$ and $r+v$ from the vertex of angle $\theta$. Writing $d_H(p,q)$ and $d_T(\varphi(p),\varphi(q))$ for the respective distances between these points, we have
\[
\begin{aligned}
\cosh d_H(p,q)
&=
\cosh(u-v)+(\cosh h-1)\cosh u\cosh v,\\
\cosh d_T(\varphi(p),\varphi(q))
&=
\cosh(u-v)+(1-\cos\theta)
                 \sinh(r+u)\sinh(r+v).
\end{aligned}
\]
Since $\sinh(r+u)\le e^r\cosh u$, equation \eqref{eq:trunc} implies $d_T(\varphi(p),\varphi(q))\le d_H(p,q)$. The same argument applies to every pair of distinct alternating sides, while distances between points on the same side are unchanged.

Now, take any closed geodesic $\gamma$ on $X_1$ cut it at its intersection points with the glued sides of the hexagons. This cuts $\gamma$ into a sequence of geodesic segments on the hexagons. By taking the curve on $X$ obtained by concatenating the same sequence of segments on the truncated triangles gives a closed path, homotopic to $\Phi(\gamma)$. Now, by applying above inequality, we have $\ell_{X_1}(\gamma) \geq \ell_X(\Phi(\gamma))$ whenever $\Phi(\gamma)$ is non-trivial. 

We now estimate the boundary length. Equation \eqref{eq:trunc} gives
\begin{equation}\label{eq:boundary}
h\geq 2\arcsinh\left(e^{r_0}\sin\frac{\theta}{2}\right).
\end{equation}

Set $x=\theta/2$ and $t=\sinh r_0$, and 

Lift the triangle to $\Hyp$. Extend the two sides of the triangle which meet in angle $\theta$ to the boundary of $\Hyp$, and consider the geodesic $\delta$ joining the corresponding ideal endpoints. Now, by considering the angle bisector which meets $\delta$ at a right angle and by using hyperbolic trigonometry in a right-angled triangle, the distance between $\delta$ and the vertex is $\arcsinh\left(\cot\left(\frac{\theta}{2}\right)\right)$. Now the original opposite side to $\theta$ meets the bisector before this ideal geodesic and, since the radius-$r_0$ corner sector does not meet the opposite side, it follows
that
\[
r_0\leq \arcsinh\left(\cot\left(\frac{\theta}{2}\right)\right).
\]
Hence, 
\[
\sinh(r_0)\, \frac{\theta}{2} \leq \frac{\theta}{2} \cot\left(\frac{\theta}{2}\right).
\]
Now using the fact that$\sinh$ is convex on $[0,1]$, and since 
$\sinh 1<4/\pi$, $\sin \left(\frac{\theta}{2}\right) \geq \frac{\theta}{\pi}$ and $e^{r_0}\geq 2 \sinh r_0$, we have
\[
\sinh \left( \frac{\theta}{2} \sinh r_0 \right) \leq \frac{\theta}{2} \sinh r_0 \sinh 1 \leq \frac{4}{\pi} \frac{\theta}{2} \sinh r_0 \leq e^{r_0} \sin  \left( \frac{\theta}{2}\right).
\]
From this, by taking the $\arcsinh$ on both sides and applying equation \eqref{eq:boundary}, we obtain
\[
h\geq  \theta\sinh r_0.
\]
We can now sum over all corners of the triangulation, whose angles sum to $2\pi$, and compare this to the sum of all $h$s, which sum to the boundary length of $X_1$. Hence
\[
\ell_{X_1}(\partial X_1) \geq 2\pi\sinh r_0.
\]
Finally, as all geodesics on $X_1$ are at least as long at their images under the forgetful map on $X$, the interior systole of $X_1$ is at least the systole of $X$. If the systole of $X_1$ has trivial image under the forgetful map, then, since it is simple, it is in fact exactly the boundary geodesic and as such it has length is at least $2\pi\sinh r_0$. \end{proof}

We note that, since we have chosen an explicit value for $r_0 (= \frac{1}{2} \arcsinh(2/\sqrt{3}))$, we have an explicit value for $2\pi\sinh r_0 (\approx 3.23)$ but the proof doesn't use this at all. In fact, it holds for any $r>0$ such that there exists an embedded disk of radius $r$ on $X$, and a construction of a companion surface $X_1$ by truncating a triangulation and building $X_1$ from the corresponding hexagons.

Putting these results together, we can prove the following, which is a restatement of Theorem \ref{thm:closedintro} from the introduction. 

\begin{theorem}[Closed surface bound]\label{thm:closed}
Let $X$ be a closed hyperbolic surface with $|\chi(X)|=q$ and $s'_0=\min\{\sys(X), 2\pi\sinh r_0\}$. Then, for any $L>0$,
\[
N_X(L,K) \leq
\left(
\frac{5}{\left(s'_0\right)^2}
e^{\frac{16q+4\pi(q+1)}{s'_0}}\,L+4
\right)^{12(q+1)}
e^{2\pi\sqrt{(q+1)K}}.
\]
\end{theorem}

\begin{proof}
Taking $X$, we construct $X_1$ as above. The curves on $X$ are sent by the map $\rho$ to curves on $X_1$ whose images have lengths $L'$ bounded by Lemma \ref{lem:injection}. Furthermore, by Lemma \ref{lem:systole}, the systole of $X_1$ is at least 
\[
\min\{s_0,2\pi\sinh r_0\}.
\]
An upper bound, for $X_1$, on the number of curves of length at most $L'$ is provided by Theorem \ref{thm:boundary}. Putting these results together gives the result. 
\end{proof}

\section{Corollaries and lower bounds}\label{sec:corollaries}

Following Sapir \cite{Sapir2021} and Torkaman \cite{Torkaman}, we now focus on geodesics of length at most $L$ and at most $\varepsilon L^2$ geodesics. By setting $K= \varepsilon L^2$, we get immediate corollaries to the previous bounds, which hold for multicurves. 

For a hyperbolic surface $X$, possibly with geodesic boundary, we set $q=|\chi(X)|$ and 
\[
h_X(\varepsilon) := \limsup_{L \to \infty}\frac{ \log\left(N_X(L,\varepsilon L^2)\right)}{L}.
\]
The previous bounds immediately imply the following results. 
\begin{corollary}\label{cor:epsilon}
When $X$ has geodesic boundary, we have \[
h_X(\varepsilon)\leq 2\pi\sqrt{q \varepsilon}.
\]
In particular, 
\[
\limsup_{\varepsilon \to 0^+} \frac{h_X(\varepsilon)}{\sqrt{\varepsilon}} \leq 2\pi\sqrt{q}.
\]
For closed $X$, we have 
When $X$ has geodesic boundary, we have \[
h_X(\varepsilon)\leq 2\pi\sqrt{(q+1) \varepsilon}
\]
and
\[
\limsup_{\varepsilon \to 0^+} \frac{h_X(\varepsilon)}{\sqrt{\varepsilon}} \leq 2\pi\sqrt{q+1}.
\]
\end{corollary}

For context, we provide a quick lower bound to the above quantity, which follows directly from Lalley's theorem \cite{Lalley}, and is in the spirit of \cite{AougabSouto}.

\begin{proposition}\label{prop:lower}
For any closed hyperbolic surface $X$ and any $\varepsilon>0$, we have 
\[
h_X(\varepsilon)\geq \min\{1,\pi\sqrt{2q\varepsilon}\}.
\]
In particular, 
\[
\liminf_{\varepsilon \to 0^+} \frac{h_X(\varepsilon)}{\sqrt{\varepsilon}} \geq \pi\sqrt{2q}.
\]
\end{proposition}

\begin{proof}
The proof follows by just proving there are enough closed geodesics with the desired intersection property. By Lalley's theorem, most closed geodesics of length at most $T$ have self-intersection approximately $\kappa T^2$ where
\[
\kappa = \frac{1}{\pi \, \area(X)} = \frac{1}{2\pi^2 q}.
\]
In particular, for any $t>0$, the asymptotic growth of curves $G_t(T)$ of length at most $T$ with at most $(t+\kappa) L^2$ intersection points grows like the total number of curves in Huber's theorem \cite{Huber}:
\[
\#G_t(T) = \#\{ \gamma \in \curves(X) \mid \ell_X(\gamma) \leq T \text{ and } \frac{1}{2} i(\gamma,\gamma) \leq (\kappa + t) L^2\} \sim \frac{e^T}{T}
\]
as $T \to \infty$. 

We now look at shorter lengths to let the intersections be close to $\varepsilon$ times the length squared. For that, we fix $t>0$ and $\varepsilon >0$ and we set 
\[
\lambda = \min \left\{ 1 , \sqrt{\frac{\varepsilon}{\kappa + t}}\right\}
\]
and $T= \lambda  L$. Any $\gamma \in G_t(T)$ satisfies $\ell_X(\gamma) \leq T \leq L$ and
\[
\frac{1}{2} i(\gamma,\gamma) \leq (\kappa+t) T^2 \leq \varepsilon L^2. 
\]
Hence, $N_X(L, \varepsilon L^2) \geq \# G_t(T) = e^{\lambda L +o(L)}$. By taking logarithms and letting $L\to \infty$ we get 
\[
h_X(\varepsilon) \geq  \lambda = \min \left\{ 1 , \sqrt{\frac{\varepsilon}{\kappa + t}}\right\}. 
\]
Now letting $t \to 0^{+}$, this gives 
\[
h_X(\varepsilon) \geq \min \left\{ 1 , \sqrt{\frac{\varepsilon}{\kappa}} \right\} = \min  \left\{ 1 , \sqrt{2q \varepsilon} \right\}
\]
as required. 
\end{proof}

These bounds on $h_X(\varepsilon)$ lead to natural questions, namely does $ \lim_{\varepsilon \to 0^+} \frac{h_X(\varepsilon)}{\sqrt{\varepsilon}}$ exist and if so what is it equal to? These questions seem intimately related to the Aougab-Souto conjecture \cite{AougabSouto} on the number of topological types of $K$-curves, one of the original motivations for this paper.

{\it Address:}\\
Department of Mathematics, University of Fribourg, Switzerland\\

{\it Email:}\\
hugo.parlier@unifr.ch\\


\begin{thebibliography}{99}

\bibitem{Andrews}
G.~E.~Andrews, \emph{The Theory of Partitions}, Encyclopedia of Mathematics
and its Applications, vol.~2, Addison--Wesley, 1976.

\bibitem{AougabSouto}
T.~Aougab and J.~Souto,
\emph{Counting curve types},
Amer. J. Math. \textbf{140} (2018), no.~6, 1423--1441.

\bibitem{BasmajianParlier2026}
A.~Basmajian and H.~Parlier,
\emph{The entropy spectrum of hyperbolic surfaces},
arXiv:2608.20143.

\bibitem{BirmanSeries84}
J.~S.~Birman and C.~Series,
\emph{An algorithm for simple curves on surfaces},
J. London Math. Soc. (2) \textbf{29} (1984), no.~2, 331--342.

\bibitem{Birman-Series85}
J. S. Birman and C. Series.
\emph{Geodesics with bounded intersection number on surfaces are sparsely distributed}.
Topology \textbf{24} (1985), no.~2, 217--225.

\bibitem{Bonahon}
F.~Bonahon,
\emph{The geometry of Teichm\"uller space via geodesic currents},
Invent. Math. \textbf{92} (1988), no.~1, 139--162.

\bibitem{BuserBook}
P.~Buser,
\emph{Geometry and spectra of compact Riemann surfaces},
reprint of the 1992 edition,
Modern Birkh\"auser Classics,
Birkh\"auser Boston, Boston, MA, 2010.

\bibitem{CahnFanoniPetri}
P.~Cahn, F.~Fanoni and B.~Petri,
\emph{Mapping class group orbits of curves with self-intersections},
Israel J. Math. \textbf{223} (2018), no.~1, 53--74.

\bibitem{Chas2004}
M.~Chas,
\emph{Combinatorial Lie bialgebras of curves on surfaces},
Topology \textbf{43} (2004), no.~3, 543--568.

\bibitem{ChasLalley}
M.~Chas and S.~P.~Lalley,
\emph{Self-intersections in combinatorial topology: statistical structure},
Invent. Math. \textbf{188} (2012), no.~2, 429--463.

\bibitem{CohenLustig}
M.~Cohen and M.~Lustig,
\emph{Paths of geodesics and geometric intersection numbers. I},
in Combinatorial group theory and topology (Alta, Utah, 1984),
Ann. of Math. Stud. \textbf{111} (1987), 479--500.

\bibitem{DeBlois18}
J.~DeBlois,
The Delaunay tessellation in hyperbolic space,
\emph{Math. Proc. Cambridge Philos. Soc.}
\textbf{164} (2018), no.~1, 15--46.

\bibitem{DGZZ}
V.~Delecroix, \'E.~Goujard, P.~Zograf and A.~Zorich,
\emph{Masur--Veech volumes, frequencies of simple closed geodesics and
intersection numbers of moduli spaces of curves},
Duke Math. J. \textbf{170} (2021), no.~12, 2633--2718.

\bibitem{DespreLazarus19}
V.~Despr\'e and F.~Lazarus,
\emph{Computing the geometric intersection number of curves},
J. ACM \textbf{66} (2019), no.~6, 1--49.

\bibitem{DespreSchlenkerTeillaud24}
V.~Despr\'e, J.-M.~Schlenker and M.~Teillaud,
Flipping geometric triangulations on hyperbolic surfaces,
\emph{J. Comput. Geom.}
\textbf{15} (2024), no.~1, 203--223.

\bibitem{ErlandssonSoutoCounting}
V.~Erlandsson and J.~Souto,
\emph{Counting curves in hyperbolic surfaces},
Geom. Funct. Anal. \textbf{26} (2016), no.~3, 729--777.

\bibitem{ErlandssonSoutoBook}
V.~Erlandsson and J.~Souto,
\emph{Mirzakhani's Curve Counting and Geodesic Currents},
Progress in Mathematics, vol.~345,
Birkh\"auser, 2022.

\bibitem{ErlandssonSoutoSimpleToAll}
V.~Erlandsson and J.~Souto,
\emph{Mirzakhani's curve counting: from simple to all},
preprint, arXiv:1904.05091.


\bibitem{FarbMargalit12}
B.~Farb and D.~Margalit,
\emph{A Primer on Mapping Class Groups},
Princeton Mathematical Series, vol.~49,
Princeton University Press, Princeton, NJ, 2012.

\bibitem{HassScott85}
J.~Hass and P.~Scott,
Intersections of curves on surfaces,
\emph{Israel J. Math.}
\textbf{51} (1985), no.~1--2, 90--120.

\bibitem{Huber}
H.~Huber,
\emph{Zur analytischen Theorie hyperbolischer Raumformen und Bewegungsgruppen},
Math. Ann. \textbf{138} (1959), 1--26.

\bibitem{Lalley}
S.~P.~Lalley,
\emph{Distribution of periodic orbits of symbolic and Axiom A flows},
Adv. in Appl. Math. \textbf{8} (1987), no.~2, 154--193.

\bibitem{LiuRafiSoutoTrin}
M.~Liu, K.~Rafi, J.~Souto and M.~Trin,
\emph{Large genus asymptotics for frequency of non-simple curves},
arXiv:2605.03833.

\bibitem{Mirzakhani2008}
M.~Mirzakhani,
\emph{Growth of the number of simple closed geodesics on hyperbolic surfaces},
Ann. of Math. (2) \textbf{168} (2008), no.~1, 97--125.

\bibitem{MirzakhaniOrbits}
M.~Mirzakhani,
\emph{Counting mapping class group orbits on hyperbolic surfaces},
preprint, arXiv:1601.03342, 2016.

\bibitem{RafiSouto}
K.~Rafi and J.~Souto,
\emph{Geodesic currents and counting problems},
Geom. Funct. Anal. \textbf{29} (2019), no.~3, 871--889.

\bibitem{Sapir2016}
J.~Sapir,
Bounds on the number of non-simple closed geodesics on a surface,
\emph{Int. Math. Res. Not. IMRN} \textbf{2016} (2016), no.~24, 7499--7545.

\bibitem{Sapir2021}
J.~Sapir, A Birman--Series type result for geodesics with infinitely many self-intersections,
\emph{Trans. Amer. Math. Soc.} \textbf{374} (2021), no.~11, 7553--7568.

\bibitem{Torkaman}
T.~Torkaman,
\emph{Entropy and self-intersection number of geodesic currents on compact hyperbolic surfaces},
arXiv:2604.05174.

\bibitem{Yamada82}
A.~Yamada,
\emph{On Marden's universal constant of Fuchsian groups. II},
J. Analyse Math. \textbf{41} (1982), 234--248.

\end{thebibliography}
\end{document}